\documentclass[a4paper,11pt]{article}

\usepackage{geometry}
\usepackage{lmodern}

\usepackage{amsmath}
\usepackage{amssymb}
\usepackage{amsthm}
\usepackage{mathtools}
\usepackage{thmtools}
\allowdisplaybreaks           

\usepackage{graphicx}
\usepackage{tikz}
\usetikzlibrary{arrows.meta,calc,positioning,decorations.pathreplacing,shapes}
\usepackage{caption}
\usepackage{subcaption}

\usepackage[T1]{fontenc}
\usepackage{microtype}
\usepackage[dvipsnames]{xcolor}
\usepackage[shortlabels]{enumitem}

\newcommand{\abs}[1]{\left\lvert#1\right\rvert}

\newcommand{\Vtwo}{V_2}                       
\newcommand{\Vthree}{V_3}                     
\newcommand{\symdiff}{\mathbin{\triangle}}    
\newcommand{\dcup}{\mathbin{\dot\cup}}        
\newcommand{\empt}{\varnothing}               
\DeclareMathOperator{\cyc}{\rho}              

\usepackage[colorlinks=true,linkcolor=BrickRed,citecolor=RoyalBlue,%
            urlcolor=RoyalBlue,linktocpage=true]{hyperref}
\usepackage[noabbrev,capitalize]{cleveref}
\crefname{equation}{}{}       
\Crefname{equation}{}{}

\numberwithin{equation}{section}

\declaretheorem[name=Theorem,numberwithin=section,refname={Theorem,Theorems},Refname={Theorem,Theorems}]{theorem}

\declaretheorem[name=Lemma,sibling=theorem,refname={Lemma,Lemmas},Refname={Lemma,Lemmas}]{lemma}

\declaretheorem[name=Corollary,sibling=theorem,refname={Corollary,Corollaries},Refname={Corollary,Corollaries}]{corollary}
\declaretheorem[name=Conjecture,sibling=theorem,refname={Conjecture,Conjectures},Refname={Conjecture,Conjectures}]{conjecture}

\declaretheorem[style=definition,name=Question,sibling=theorem,refname={Question,Questions},Refname={Question,Questions}]{question}

\crefname{theorem}{Theorem}{Theorems}
\Crefname{theorem}{Theorem}{Theorems}
\crefname{lemma}{Lemma}{Lemmas}
\Crefname{lemma}{Lemma}{Lemmas}
\crefname{proposition}{Proposition}{Propositions}
\Crefname{proposition}{Proposition}{Propositions}
\crefname{corollary}{Corollary}{Corollaries}
\Crefname{corollary}{Corollary}{Corollaries}
\crefname{conjecture}{Conjecture}{Conjectures}
\Crefname{conjecture}{Conjecture}{Conjectures}
\crefname{definition}{Definition}{Definitions}
\Crefname{definition}{Definition}{Definitions}
\crefname{remark}{Remark}{Remarks}
\Crefname{remark}{Remark}{Remarks}
\crefname{question}{Question}{Questions}
\Crefname{question}{Question}{Questions}

\tikzset{
  cvertex/.style = {circle, draw=black, fill=black, inner sep=0pt, minimum size=1.6mm},
  tvertex/.style = {circle, draw=black, fill=white, line width=0.7pt, inner sep=0pt, minimum size=2.2mm},
  nvertex/.style = {circle, draw=black, fill=black!12, inner sep=0pt, minimum size=2.0mm},
  gedge/.style   = {line width=0.8pt},
  medge/.style   = {line width=1.7pt, line cap=round},
  redge/.style   = {line width=0.8pt, dashed},
  qregion/.style = {draw=black!45, dashed, fill=black!5, rounded corners=6pt},
  glabel/.style  = {font=\footnotesize},
}
\tikzset{>=latex}

\newcommand{\transarrow}{\ensuremath{\longrightarrow}}

\title{A proof of the 3-Decomposition Conjecture}

\author{Jicheng Ma\thanks{Corresponding author. School of Mathematics, Renmin
University of China, Beijing, 100872, China. Email:
\texttt{mjc191812@ruc.edu.cn}.}}

\date{}

\begin{document}

\maketitle

\begin{abstract}
We prove the $3$-Decomposition Conjecture: every finite connected cubic
loopless multigraph decomposes into a spanning tree, a $2$-regular subgraph, and
a matching. The proof rests on a new theorem on matching complements in
subcubic graphs. Let $H$ be a finite connected bridgeless simple graph of
maximum degree three, and let $S$ be its set of degree-two vertices, with
$\abs{S}=k\ge 2$. We show that $H$ has a matching of size equal to its
cyclomatic number, $\abs{E(H)}-\abs{V(H)}+1$, whose deletion leaves a single
tree containing all of $S$, together with cycles disjoint from $S$. The proof
is by induction on~$k$.
\end{abstract}


\smallskip


\bigskip

\section{Introduction}\label{sec:intro}

Decompositions of cubic graphs into structured spanning subgraphs have attracted
sustained attention since the work of Petersen~\cite{Petersen1891}.
A \emph{$3$-decomposition} of a graph $G$ is a partition of its edge set into a
spanning tree, a $2$-regular subgraph, and a matching. By Petersen's
theorem, every bridgeless cubic graph has a perfect matching,
and hence a $2$-factor. Malkevitch~\cite{Malkevitch1979} studied which cubic
graphs decompose into a spanning tree and a $2$-regular subgraph; for a cubic
graph this is equivalent to having a homeomorphically irreducible spanning tree
(HIST), that is, a spanning tree with no vertex of degree two. Allowing a
matching as a third part, Hoffmann-Ostenhof~\cite{HoffmannOstenhof2011}
conjectured in 2011 that such a decomposition always exists.

\begin{conjecture}[3-Decomposition Conjecture~\cite{HoffmannOstenhof2011}]\label{conj:3dc}
  Every finite connected cubic loopless multigraph has an edge partition into a
  spanning tree, a $2$-regular subgraph, and a matching.
\end{conjecture}

\Cref{conj:3dc} has been open for fifteen years and has become one of the
central open problems on edge decompositions of cubic graphs.
Hoffmann-Ostenhof, Kaiser, and Ozeki~\cite{HoffmannOstenhofKaiserOzeki2018}
showed that it is equivalent to the \emph{$2$-Decomposition Conjecture} (2DC):
every connected loopless multigraph with all degrees two or three, in which
every cycle is separating, decomposes into a spanning tree and a matching.
Botler, Jim\'enez, Sambinelli, and
Wakabayashi~\cite{BotlerJimenezSambinelliWakabayashi2024} established strong
structural restrictions on a minimum counterexample to 2DC, proving in
particular that it must be simple and $2$-edge-connected.

Several results relax the matching to a union of vertex-disjoint paths of length
at most two. Hong, Liu, and Yu~\cite{HongLiuYu2020}, Li and
Cui~\cite{LiCui2014}, and Lyngsie and Merker~\cite{LyngsieMerker2019}
independently established such a decomposition for every connected cubic graph,
the last as a specialization of a decomposition into a spanning tree, an even
graph, and a star forest (see also Akbari, Jensen, and
Siggers~\cite{AkbariJensenSiggers2015} for related decompositions into forests
and regular subgraphs). Fan and Zhou~\cite{FanZhou2025} then bounded the number
of length-two paths by $(n-4)/6$, and Fan, Guo, and Zhou~\cite{FanGuoZhou2026}
improved this to $(n-4)/8$ (\cref{conj:3dc} corresponds to having no paths of
length two). Zhang and Szeider~\cite{ZhangSzeider2025} verified the conjecture by
computer for all relevant graphs up to $28$ vertices.

The conjecture has also been confirmed for several families of graphs. The
planar case was settled by Hoffmann-Ostenhof, Kaiser, and
Ozeki~\cite{HoffmannOstenhofKaiserOzeki2018}, following the $3$-connected planar
case of Ozeki and Ye~\cite{OzekiYe2016}; it is known for traceable cubic
graphs~\cite{AbdolhosseiniAkbariHashemiMoradian2016, LiuLi2020}, for claw-free
or $4$-chordal subcubic graphs~\cite{AboomahigirAhanjidehAkbari2018}, for cubic
graphs having a $2$-factor consisting of three cycles~\cite{XieZhouZhou2020},
and for $3$-connected star-like graphs~\cite{BachtlerKrumke2022}. Bachtler and
Heinrich~\cite{BachtlerHeinrich2023} identified several reducible
configurations, relating the problem to the HIST-extension conjecture, and
proved the conjecture for $3$-connected cubic graphs of path-width at most
four. Related matching-decycling problems were studied by Protti and
Souza~\cite{ProttiSouza2018}.

In this paper, We prove the conjecture~\ref{conj:3dc}.

\begin{theorem}\label{thm:3dc-resolved}
  The $3$-Decomposition Conjecture holds.
\end{theorem}

The main difficulty in attacking the $2$- and $3$-Decomposition Conjectures
directly by induction is the fragility condition: requiring every cycle to be
separating is a global property that is not easily preserved under local
reductions. To bypass this difficulty, we drop fragility during the induction
and work instead in the class of all connected bridgeless simple graphs with
degrees two and three. Call a graph in which every vertex has degree two or
three a \emph{$\{2,3\}$-graph}, and a degree-two vertex of such a graph a
\emph{terminal}; write $\Vtwo(H)$ for its set of terminals and $\Vthree(H)$ for
its set of degree-three vertices. Our main technical result is the following.

\begin{theorem}\label{thm:terminal-selection}
  Let $H$ be a finite connected bridgeless simple graph, every vertex of which
  has degree two or three. If $S=\Vtwo(H)$ with $\abs{S}=k\ge 2$, then $H$ has a
  matching $M$ of size
  \begin{equation}\label{eq:main-size}
    \abs{M}=q:=\frac{\abs{V(H)}-k+2}{2}
  \end{equation}
  such that $H-M$ is the disjoint union of one tree containing all vertices of
  $S$ and zero or more cycles.
\end{theorem}

The size in \cref{eq:main-size} equals the cyclomatic number
$\abs{E(H)}-\abs{V(H)}+1$, the only value for which the complement can be a
single tree together with cycles. The hypothesis $k\ge 2$ is necessary: when
$k=0$ or $k=1$ one has $q>\abs{V(H)}/2$, so no matching of the prescribed size
exists.

A cycle $C$ in a connected graph $G$ is \emph{separating} if $G-E(C)$ is
disconnected, and $G$ is \emph{fragile} when every one of its cycles is
separating. In a fragile graph, a removable-cycle theorem of Borse and Waphare
(\cref{thm:protected-cycle}) rules out any residual cycle disjoint from the
tree component containing all terminals. \Cref{thm:terminal-selection}
therefore yields the following consequence, proved in \cref{sec:proof-3dc}.

\begin{theorem}\label{thm:simple-fragile}
  Every finite connected simple fragile subcubic graph admits an edge partition
  into a spanning tree and a matching.
\end{theorem}

\Cref{thm:simple-fragile} extends directly to loopless multigraphs by
induction on the number of parallel edges.

\begin{corollary}\label{cor:multigraph-fragile}
  Every finite connected fragile subcubic loopless multigraph admits an edge
  partition into a spanning tree and a matching.
\end{corollary}

We prove \cref{thm:terminal-selection} by induction on the number
of terminals. A degree count shows that a matching of the prescribed
size has the required complement whenever all terminals lie in one
component. Two local reductions allow us to suppress a terminal
while preserving simplicity. An auxiliary cubic graph provides
a matching $N$ covering every degree-three vertex and at least one
terminal. We suppress a terminal covered by $N$ and apply the
induction hypothesis. Restoring the terminal either gives the desired
matching or leaves one exposed vertex outside the component containing
the other terminals. In the latter case, an alternating-path exchange
covers this vertex and exposes a vertex of the original tree component.
The parity identity then forces all terminals into one component
of the new complement.

The paper is organized as follows. \Cref{sec:prelim} fixes conventions and
recalls two external results: Sch\"onberger's theorem on perfect matchings and
a removable-cycle theorem of Borse and Waphare, for which we include a short
self-contained proof. \Cref{sec:terminal-selection} proves
\cref{thm:terminal-selection}: the preliminary lemmas are collected in
\cref{sec:counting,sec:normalizations,sec:cubic-tree}, and the induction is
completed in \cref{sec:first-entry}. In \cref{sec:proof-3dc} we deduce
\cref{thm:simple-fragile} and \cref{cor:multigraph-fragile}, and derive 2DC
and 3DC. \Cref{sec:scope} concludes with open problems.

\section{Preliminaries}\label{sec:prelim}

All graphs in this paper are finite. In a loopless multigraph, degrees count
parallel edges with multiplicity, and a pair of parallel edges forms a cycle of
length two. Given a graph $G$ and a set $F\subseteq E(G)$, we write $G-F$ for
the spanning subgraph $(V(G),E(G)\setminus F)$. For a matching $M$, a vertex is
\emph{covered} if it is incident with an edge of $M$, and \emph{exposed}
otherwise. A $2$-regular subgraph need not be spanning and may be empty. For a
connected graph $L$, we write $\cyc(L)=\abs{E(L)}-\abs{V(L)}+1$ for its
cyclomatic number, and we write $A\dcup B$ to denote the union of two disjoint
sets $A$ and $B$.

\begin{theorem}[Sch\"onberger~\cite{Schonberger1935}]\label{thm:specified-edge}
  If $K$ is a finite simple bridgeless cubic graph and $e\in E(K)$, then $K$
  has a perfect matching containing $e$.
\end{theorem}

This strengthening of Petersen's theorem is due to
Sch\"onberger~\cite{Schonberger1935} (see also Boyd, Iwata, and
Takazawa~\cite[Theorem~2.1]{BoydIwataTakazawa2013} for the loopless-multigraph
version). In particular, every finite simple bridgeless cubic graph has a
perfect matching.

Our second external tool provides a removable cycle avoiding a prescribed
connected subgraph (compare the classical results on non-separating cycles in
graphs~\cite{ThomassenToft1981,Tutte1963}).

\begin{theorem}[Borse--Waphare~\cite{BorseWaphare2013}]\label{thm:protected-cycle}
  Let $G$ be a connected simple graph and let $Q$ be a connected subgraph.
  Suppose that $G$ contains a cycle edge-disjoint from $Q$, and that every
  $v\in V(G)\setminus V(Q)$ satisfies $d_G(v)\ge 3$. Then $G$ contains a cycle
  $C$ edge-disjoint from $Q$ such that $G-E(C)$ is connected.
\end{theorem}

This theorem is due to Borse and Waphare~\cite[Theorem~2]{BorseWaphare2013}.
We include a proof for completeness.

\begin{proof}
  Let $\mathcal H$ be the family of connected subgraphs $H$ of $G$ such that
  $Q\subseteq H$ and the spanning subgraph $F_H:=G-E(H)$ contains a cycle. By
  assumption, $G$ contains a cycle edge-disjoint from $Q$, so $Q\in\mathcal H$.
  Since $G$ is finite, choose $H\in\mathcal H$ with $\abs{E(H)}$ maximum, and
  write $F=F_H$.

  Every component of $F$ meets $V(H)$. Otherwise a shortest $G$-path from such
  a component to $V(H)$ would, up to its first vertex in $V(H)$, use only edges
  outside $E(H)$; it would therefore be an $F$-path joining the component to
  $V(H)$, a contradiction.

  Choose a component of $F$ that contains a cycle. It meets $V(H)$ and is
  nontrivial, so it has an edge $e\in E(F)$ incident with a vertex of $H$. Let
  $H+e$ be the subgraph of $G$ obtained from $H$ by adding $e$ and its
  endpoints; then $H+e$ is connected and contains $Q$. Maximality of $H$ forces
  $F-e$ to be acyclic, for otherwise $H+e$ would be a member of $\mathcal H$
  with more edges. Thus $F-e$ is a forest. Since $F$ contains a cycle, adding
  the single edge $e$ to that forest creates exactly one cycle, which we denote
  by $C$. In particular, $C$ is edge-disjoint from $H$, and hence from $Q$.

  It remains to prove that $G-E(C)$ is connected. Suppose a component $D$ of
  the forest $F-E(C)$ does not meet $V(H)$. The component of $F$ containing $D$
  does meet $V(H)$, so $D$ contains a vertex of $C$; otherwise $D$ would
  already be an $F$-component. It contains at most one such vertex: if there
  were two or more, choose two at minimum distance in $D$. Their $D$-path has
  no internal vertex on $C$, and together with either arc of $C$ it would give
  a second cycle in $F$. Consequently
  \[
    V(D)\cap V(C)=\{x\}
  \]
  for some vertex $x$.

  Every vertex of $D$ lies outside $V(H)$, hence outside $V(Q)$, and every edge
  of $G$ incident with such a vertex belongs to $F$. No edge of $C$ is incident
  with a vertex of $V(D)\setminus\{x\}$, while at $x$ exactly the two edges of
  $C$ incident with $x$ were deleted. Therefore
  \[
    d_D(v)=d_G(v)\ge 3\quad(v\in V(D)\setminus\{x\}),
    \qquad
    d_D(x)=d_G(x)-2\ge 1.
  \]
  Writing $m=\abs{V(D)}$ and using that $D$ is a finite tree gives
  \[
    2(m-1)=\sum_{v\in V(D)}d_D(v)\ge 1+3(m-1)=3m-2,
  \]
  which forces $m\le 0$, a contradiction. Hence every component of $F-E(C)$
  meets the connected subgraph $H$. Their union with $H$ is precisely
  $G-E(C)$, and is connected. This proves the theorem.
\end{proof}

\section{Proof of \texorpdfstring{Theorem~\ref{thm:terminal-selection}}{Theorem 1.3}}\label{sec:terminal-selection}

We prove \cref{thm:terminal-selection} by induction on the number $k$ of
terminals. We first collect the lemmas needed for the induction step; the
induction itself is carried out in \cref{sec:first-entry}.

\subsection{Preliminary lemmas}\label{sec:counting}

Let $H$ be a $\{2,3\}$-graph, put $n=\abs{V(H)}$, and let $k=\abs{\Vtwo(H)}$.
The handshaking identity gives
\[
  2\abs{E(H)}=2k+3(n-k)=3n-k.
\]
Consequently $n-k=\abs{\Vthree(H)}$ is even and
\[
  \abs{E(H)}-\abs{V(H)}+1=\frac{n-k+2}{2}=\frac{\abs{\Vthree(H)}}{2}+1.
\]
Thus the quantity
\[
  q:=\frac{\abs{V(H)}-k+2}{2}
\]
is an integer and $q\ge 1$. A matching of size $q$ covers $2q=n-k+2$ vertices,
and therefore exposes exactly
\begin{equation}\label{eq:exposed}
  n-2q=k-2
\end{equation}
vertices.

Let $M$ be any matching of $H$, and let $L$ be a component of $H-M$. Define
\[
  t(L)=\abs{V(L)\cap \Vtwo(H)},
  \qquad
  z(L)=\abs{\{v\in V(L):v\text{ is exposed by }M\}},
\]
and recall the cyclomatic number $\cyc(L)=\abs{E(L)}-\abs{V(L)}+1$.

\begin{lemma}\label{lem:component-identity}
  Let $H$ be a $\{2,3\}$-graph and $M$ be a matching of $H$. For every
  component $L$ of $H-M$,
  \begin{equation}\label{eq:component-identity}
    z(L)-t(L)=2\cyc(L)-2.
  \end{equation}
\end{lemma}

\begin{proof}
  In $H-M$, a covered terminal has degree one, an exposed terminal has degree
  two, a covered degree-three vertex has degree two, and an exposed
  degree-three vertex has degree three. Hence the sum of $d_L(v)-2$ over
  $v\in V(L)$ equals the number of exposed degree-three vertices in $L$ minus
  the number of covered terminals in $L$, which is $z(L)-t(L)$. On the other
  hand,
  \[
    \sum_{v\in V(L)}\bigl(d_L(v)-2\bigr)=2\abs{E(L)}-2\abs{V(L)}=2\cyc(L)-2.
  \]
  Comparing the two expressions proves \cref{eq:component-identity}.
\end{proof}

\begin{lemma}\label{lem:terminal-criterion}
  Let $H$ be a $\{2,3\}$-graph with $k\ge 2$ terminals, and let $M$ be a
  matching of $H$ of size $q=(\abs{V(H)}-k+2)/2$. Then the following are
  equivalent.
  \begin{enumerate}[(i)]
    \item All $k$ terminals lie in one component of $H-M$.
    \item The graph $H-M$ consists of one tree containing all terminals and
      zero or more cycles containing no terminal.
  \end{enumerate}
  Moreover, all $k-2$ exposed vertices then lie in the distinguished tree.
\end{lemma}

\begin{proof}
  Condition~(ii) plainly implies~(i). Conversely, let $L$ be the component
  containing all $k$ terminals. By \cref{eq:exposed}, $z(L)\le k-2$. Identity
  \cref{eq:component-identity} together with $\cyc(L)\ge 0$ gives
  $z(L)\ge k-2$. Thus $z(L)=k-2$ and $\cyc(L)=0$, so $L$ is a tree and contains
  every exposed vertex. Any other component $R$ has $t(R)=z(R)=0$, whence
  $\cyc(R)=1$. Every vertex of $R$ is a covered degree-three vertex, and hence
  has degree two in $H-M$. Therefore $R$ is a cycle.
\end{proof}

The next lemma uses an alternating path to move an exposed vertex into a given
component of a matching complement.

\begin{lemma}\label{lem:first-entry}
  Let $M,N$ be matchings of a finite simple graph $G$, let $Q$ be a component
  of $G-M$, and let $x\notin V(Q)$. Assume that
  \begin{itemize}
    \item $x$ is exposed by $M$ and covered by $N$;
    \item every other vertex exposed by $M$ belongs to $Q$; and
    \item every vertex exposed by $N$ belongs to $Q$.
  \end{itemize}
  Then there is a matching $M'$ with $\abs{M'}=\abs{M}$ such that the component
  of $G-M'$ containing $Q$ contains every vertex exposed by $M'$.
\end{lemma}

\begin{proof}
  In $M\symdiff N$, the vertex $x$ has degree one, and its incident edge belongs
  to $N\setminus M$. Its component is therefore a simple alternating path. The
  other endpoint is exposed by exactly one of $M,N$, and hence belongs to $Q$ by
  the hypotheses.

  Traverse the path from $x$, and stop at its first vertex $d\in V(Q)$. The edge
  that enters $Q$ must belong to $M\setminus N$: if it lay outside $M$, it would
  be an edge of $G-M$ joining its preceding outside vertex to the component $Q$,
  contrary to the choice of $Q$. Thus the prefix from $x$ to $d$ begins with an
  $N\setminus M$ edge and ends with an $M\setminus N$ edge, and it contains
  equally many edges of the two matchings.

  Let $M'$ be the matching obtained from $M$ by exchanging matching and
  non-matching edges along this prefix. Then $\abs{M'}=\abs{M}$, and $M'$ covers
  $x$ and exposes $d$. The prefix meets $Q$ for the first time at $d$, so no
  edge of $Q$ is removed from $G-M$, and the final $M$-edge becomes an edge of
  $G-M'$ entering $Q$. Every vertex exposed by $M$ other than $x$ remains in
  $Q$, and the only newly exposed vertex is $d\in V(Q)$. Hence every vertex
  exposed by $M'$ lies in the component of $G-M'$ containing $Q$.
\end{proof}

\subsection{Two reductions}\label{sec:normalizations}

\begin{lemma}\label{lem:reductions}
  Let $k\ge 3$, suppose that \cref{thm:terminal-selection} holds for graphs with
  $k-1$ terminals, and let $H$ be a counterexample to
  \cref{thm:terminal-selection} with $k$ terminals and the fewest possible
  vertices. Then
  \begin{enumerate}[(a)]
    \item no two terminals of $H$ are adjacent, and
    \item the two neighbours of every terminal of $H$ are nonadjacent vertices
      of degree three.
  \end{enumerate}
\end{lemma}

\begin{proof}
We first prove~(a). Suppose two terminals $u,v$ are adjacent. Let $x$ and $y$
be their respective other neighbours. If $x=y$ has degree three, its third edge
is the only edge from the triangle $uvx$ to the rest of the graph, hence a
bridge. If $x=y$ has degree two, connectedness forces $H=C_3$, for which any
one-edge matching leaves a spanning path. Neither case can occur, so $x\ne y$.

Contract $uv$ to a new vertex $w$, retaining the edges $wx,wy$
(\cref{fig:reduction-adjacent}). The resulting graph $H_1$ is connected,
simple, bridgeless, and has $k-1$ terminals. It is simple because $x\ne y$;
bridgelessness follows by contracting $uv$ inside a cycle through each surviving
edge. The prescribed matching size is unchanged,
\[
  \frac{(\abs{V(H)}-1)-(k-1)+2}{2}=q .
\]
Apply the induction hypothesis to $H_1$, and let $Q$ be its tree component
containing all terminals. Since $w$ is a terminal, at most one of $wx,wy$ lies
in the selected matching. There are three cases:
\begin{itemize}
  \item If neither edge is matched, replace the tree path $xwy$ by $xuvy$.
  \item If $wx$ is matched, replace $wx$ in the matching by $xu$, and replace
    the tree leaf $yw$ by the path $yvu$.
  \item If $wy$ is matched, use the symmetric construction.
\end{itemize}
In each case, the resulting matching of $H$ has size $q$, expands only a path
or a leaf of $Q$, and leaves every residual cycle unchanged, contradicting that
$H$ is a counterexample. This proves~(a).

\begin{figure}[ht]
  \centering
  \begin{tikzpicture}[baseline]
    \node[tvertex,label={[glabel]right:$u$}] (u) at (0,0.65) {};
    \node[tvertex,label={[glabel]right:$v$}] (v) at (0,-0.65) {};
    \node[cvertex,label={[glabel]above:$x$}] (x) at (-1.5,1.3) {};
    \node[cvertex,label={[glabel]below:$y$}] (y) at (-1.5,-1.3) {};
    \draw[medge] (u) -- (v);
    \draw[gedge] (u) -- (x);
    \draw[gedge] (v) -- (y);
    \draw[gedge] (x) -- ++(160:0.75); \draw[gedge] (x) -- ++(200:0.75);
    \draw[gedge] (y) -- ++(160:0.75); \draw[gedge] (y) -- ++(200:0.75);
  \end{tikzpicture}
  \qquad\transarrow\qquad
  \begin{tikzpicture}[baseline]
    \node[tvertex,label={[glabel]right:$w$}] (w) at (0,0) {};
    \node[cvertex,label={[glabel]above:$x$}] (x) at (-1.5,1.3) {};
    \node[cvertex,label={[glabel]below:$y$}] (y) at (-1.5,-1.3) {};
    \draw[gedge] (w) -- (x);
    \draw[gedge] (w) -- (y);
    \draw[gedge] (x) -- ++(160:0.75); \draw[gedge] (x) -- ++(200:0.75);
    \draw[gedge] (y) -- ++(160:0.75); \draw[gedge] (y) -- ++(200:0.75);
  \end{tikzpicture}
  \caption{Removing an adjacency between terminals: the edge $uv$ (heavy) is
  contracted to a single terminal $w$. Hollow nodes are terminals; solid nodes
  have degree three.}
  \label{fig:reduction-adjacent}
\end{figure}

Next we prove~(b). Let a terminal $s$ have neighbours $a,b$. By~(a), both
have degree three. Suppose that $ab\in E(H)$, and let $x$ and $y$ be the third
neighbours of $a$ and $b$, respectively. If $x=y$ has degree three, its third
edge is the sole edge leaving the four-vertex subgraph on $\{s,a,b,x\}$, hence
a bridge. If $x=y$ has degree two, connectedness makes this four-vertex
subgraph all of $H$, with exactly two terminals, contrary to $k\ge 3$.
Therefore $x\ne y$.

Replace the triangle $sab$ by a new degree-two vertex $w$ adjacent to $x,y$
(\cref{fig:reduction-triangle}). The graph $H_2$ is connected, and it is simple because $x\ne y$. Every cycle of $H$ other than $sab$ either avoids $\{s,a,b\}$
or uses both $ax$ and $by$. In the latter case, replacing its
segment through $\{s,a,b\}$ by $xwy$ gives a cycle of $H_2$.
Since every edge of $H$ outside the triangle lies on a cycle
distinct from $sab$, every edge of $H_2$ lies on a cycle.
Thus $H_2$ is bridgeless. It has $k$ terminals and two fewer vertices, so
\cref{eq:main-size} gives $q-1$ in place of $q$. By the minimality of
$\abs{V(H)}$, the graph $H_2$ has a matching of size $q-1$ whose complement
consists of a tree $Q$ containing all terminals and zero or more cycles. Again
there are three cases at $w$:
\begin{itemize}
  \item If neither $wx$ nor $wy$ is matched, add $ab$ to the matching and
    replace the tree path $xwy$ by $xasby$.
  \item If $wx$ is matched, replace it by the two disjoint matching edges
    $ax,sb$, and replace the tree leaf $yw$ by $ybas$.
  \item If $wy$ is matched, use the symmetric construction.
\end{itemize}
In each case, the matching size increases by one and only a path or a leaf of
$Q$ is expanded. This again contradicts the choice of $H$, and proves~(b).
\end{proof}

\begin{figure}[ht]
  \centering
  \begin{tikzpicture}[baseline]
    \node[tvertex,label={[glabel]above:$s$}] (s) at (0,1.05) {};
    \node[cvertex,label={[glabel]below:$a$}] (a) at (-0.8,0) {};
    \node[cvertex,label={[glabel]below:$b$}] (b) at (0.8,0) {};
    \node[cvertex,label={[glabel]above:$x$}] (x) at (-2.0,0.75) {};
    \node[cvertex,label={[glabel]above:$y$}] (y) at (2.0,0.75) {};
    \draw[gedge] (s) -- (a);
    \draw[gedge] (s) -- (b);
    \draw[gedge] (a) -- (b);
    \draw[gedge] (a) -- (x);
    \draw[gedge] (b) -- (y);
    \draw[gedge] (x) -- ++(150:0.7); \draw[gedge] (x) -- ++(215:0.7);
    \draw[gedge] (y) -- ++(30:0.7);  \draw[gedge] (y) -- ++(325:0.7);
  \end{tikzpicture}
  \qquad\transarrow\qquad
  \begin{tikzpicture}[baseline]
    \node[tvertex,label={[glabel]above:$w$}] (w) at (0,0.75) {};
    \node[cvertex,label={[glabel]above:$x$}] (x) at (-1.6,0.75) {};
    \node[cvertex,label={[glabel]above:$y$}] (y) at (1.6,0.75) {};
    \draw[gedge] (w) -- (x);
    \draw[gedge] (w) -- (y);
    \draw[gedge] (x) -- ++(150:0.7); \draw[gedge] (x) -- ++(215:0.7);
    \draw[gedge] (y) -- ++(30:0.7);  \draw[gedge] (y) -- ++(325:0.7);
  \end{tikzpicture}
  \caption{Removing a triangle at a terminal: the triangle $sab$ is replaced by
  a single terminal $w$ joined to the two outer neighbours $x,y$.}
  \label{fig:reduction-triangle}
\end{figure}

\subsection{An auxiliary cubic graph}\label{sec:cubic-tree}

\begin{lemma}\label{lem:aux-matching}
  Let $H$ be a connected bridgeless simple $\{2,3\}$-graph with $k\ge 3$
  terminals. Then $H$ has a matching that covers every vertex of degree three
  and at least one terminal.
\end{lemma}

\begin{proof}
Since $k\ge 3$, there exists a tree $\mathcal T$ with exactly $k$ leaves in
which every internal vertex has degree three (for instance, obtained from
$K_{1,3}$ by repeatedly subdividing a leaf edge and attaching a new leaf). The
handshaking identity for a tree shows that $\mathcal T$ has exactly $k-2$
internal vertices. Moreover, since $k\ge 3$ and every internal vertex has
degree three, no two leaves of $\mathcal T$ are adjacent in $\mathcal T$.

Identify the $k$ leaves of $\mathcal T$ bijectively with the $k$ terminals of
$H$, keeping the $k-2$ internal vertices of $\mathcal T$ disjoint from $V(H)$,
and denote the resulting graph by
\[
  \widehat H=H\cup\mathcal T .
\]
Because no two leaves of $\mathcal T$ are adjacent, identifying the leaves with
the terminals of $H$ creates no parallel edges, so $\widehat H$ is simple and
cubic. It is also bridgeless. Every edge of $H$ lies on a cycle of $H$. If
$g\in E(\mathcal T)$, then each component of $\mathcal T-g$ contains a leaf of
$\mathcal T$: otherwise, in a component containing no leaf of $\mathcal T$, the
endpoint of $g$ would have degree two and every other vertex would have degree
three, contradicting that every finite tree has a vertex of degree at most one.
Choose terminal leaves $u,v$ in different components of $\mathcal T-g$. A
simple $u$--$v$ path in $H$ and the unique $u$--$v$ path in $\mathcal T$ meet
only at $u$ and $v$, since all internal vertices of $\mathcal T$ lie outside
$V(H)$. Their union is a cycle of $\widehat H$ containing $g$. Thus every edge
of $\widehat H$ lies on a cycle.

Choose a perfect matching $P$ of $\widehat H$ by \cref{thm:specified-edge}. Let
$I$ be the set of the $k-2$ new internal vertices, and set
\[
  r=\abs{P\cap E(\mathcal T[I])},
  \qquad
  p=\abs{\{e\in P:e\text{ joins a terminal to }I\}} .
\]
Counting how $P$ covers the internal vertices gives
\[
  2r+p=k-2 .
\]

Let $N=P\cap E(H)$. Every vertex of degree three in $H$ is covered by $N$,
since all its edges in $\widehat H$ belong to $H$. A terminal is exposed by $N$
exactly when its edge to $I$ belongs to $P$. Hence exactly $p$ terminals are
exposed, and $N$ covers $k-p=2+2r\ge 2$ terminals.
\end{proof}

\subsection{Completing the proof}\label{sec:first-entry}

\begin{proof}[Proof of \cref{thm:terminal-selection}]
We argue by induction on $k$. We first treat the base case $k=2$; let
$S=\{s,t\}$. Take four new vertices $a,b,c,d$, join them by the five edges
$ac,ad,bc,bd,cd$, a copy of $K_4$ minus the edge $ab$, and add the two edges
$sa$ and $tb$. The resulting graph $K$ is
simple and cubic. It is also bridgeless: every edge of $H$ lies on a cycle of
$H$, each of the five edges among $a,b,c,d$ lies on one of the triangles $acd$
and $bcd$, and the edges $sa$ and $tb$ lie on the cycle formed by an $s$--$t$
path in $H$ together with the path $s\,a\,c\,b\,t$.

Choose an edge $f\in E(H)$ incident with $s$. By \cref{thm:specified-edge}, $K$
has a perfect matching $P$ containing $f$. Then $sa\notin P$, so $P$ matches
$a$ to one of $c,d$. The other vertex of $\{c,d\}$ then has $b$ as its only
uncovered neighbour, so it is matched to $b$, and therefore $tb\notin P$.
Consequently $M=P\cap E(H)$ is a perfect matching of $H$ of size $n/2=q$. In
$H-M$, the vertices $s,t$ have degree one and every other vertex has degree
two. Therefore $H-M$ consists of one $s$--$t$ path together with zero or more
cycles. This proves the theorem for $k=2$.

Now let $k\ge 3$, and assume that the theorem holds for graphs with $k-1$
terminals. Suppose that it fails for some graph with $k$ terminals, and let $H$
be such a counterexample with the fewest vertices. By \cref{lem:reductions},
the terminals of $H$ are pairwise nonadjacent, and the two neighbours of each
terminal are nonadjacent vertices of degree three.

Let $N$ be a matching of $H$ as in \cref{lem:aux-matching}, and choose a
terminal $s$ covered by $N$. Let $a$ and $b$ be the neighbours of $s$, and let
$J$ be the graph obtained from $H$ by deleting $s$ and adding the edge $e=ab$.
By \cref{lem:reductions}, $J$ is simple and remains a $\{2,3\}$-graph with
$k-1$ terminals. It is connected and bridgeless: since either incident edge at
$s$ lies on a cycle, $H-s$ contains an $a$--$b$ path; every cycle avoiding $s$
survives; a cycle using the segment $asb$ becomes a cycle using $ab$; and the
new edge $ab$ itself lies on such a cycle. The prescribed matching size for $J$
is again
\[
  \frac{(n-1)-(k-1)+2}{2}=q .
\]

Apply the induction hypothesis to $J$. Let $M_J$ be the resulting matching and
$Q$ its tree component containing all $k-1$ terminals. By
\cref{lem:terminal-criterion}, all $k-3$ vertices
exposed by $M_J$ lie in $Q$.

We now restore $s$ and the edges $as,sb$ to obtain a matching $M$ of $H$ of
size $q$, according to the position of $e$ relative to $M_J$ and $Q$
(\cref{fig:restore-states}).
\begin{enumerate}[(1)]
  \item If $e\notin M_J$ and $e\in E(Q)$, keep the matching and subdivide the
    tree edge $e$. All $k$ terminals then lie in one component, so by
    \cref{lem:terminal-criterion} the matching has the required properties, a
    contradiction.
  \item If $e\notin M_J$ and $e$ lies on a residual cycle, keep the matching.
    The old tree $Q$ is unchanged, while the restored terminal $s$ is the unique
    new exposed vertex outside $Q$.
  \item If $e\in M_J$ and at least one endpoint, say $a$, belongs to $Q$,
    replace $e$ in the matching by $sb$. The complement edge $as$ attaches $s$
    to $Q$, and the newly exposed vertex $a$ also lies in $Q$. Again all
    terminals lie in one component, and \cref{lem:terminal-criterion} gives a
    contradiction.
  \item If $e\in M_J$ and neither endpoint belongs to $Q$, replace $e$ by either
    $as$ or $sb$. The old tree $Q$ remains unchanged. The endpoint not covered
    by the replacement is the unique new exposed vertex outside $Q$, and $s$
    also remains outside $Q$.
\end{enumerate}
Each replacement preserves the matching cardinality $q$.

It remains to treat cases~(2) and~(4). In either case, let $x$ be the unique
vertex exposed by $M$ outside $Q$. Then
\begin{itemize}
  \item $Q$ contains the $k-1$ terminals other than $s$ and all other vertices
    exposed by $M$;
  \item every vertex exposed by $N$ is a terminal different from $s$, because
    $s$ was chosen to be covered by $N$, and hence lies in $Q$; and
  \item $x$ is covered by $N$: in case~(2) it is the chosen terminal $s$, and
    in case~(4) it is a degree-three vertex of $H$, all of which are covered by
    $N$.
\end{itemize}
\Cref{lem:first-entry} therefore yields a matching $M'$ with
$\abs{M'}=\abs{M}=q$ such that the component $Q'$ of $H-M'$ containing $Q$
contains all $k-2$ vertices exposed by $M'$. It also contains the $k-1$
terminals in $V(Q)$. If $s\notin V(Q')$, then \cref{eq:component-identity}
would give
\[
  z(Q')-t(Q')=(k-2)-(k-1)=-1=2\cyc(Q')-2 ,
\]
which is impossible because the right-hand side is even. Hence $Q'$ contains
all $k$ terminals, and \cref{lem:terminal-criterion} implies that $M'$ has the
required properties, contrary to the choice of $H$. This completes the
induction.
\end{proof}

\begin{figure}[ht]
  \centering
  \begin{subcaptionbox}{Case (1): $e\in E(Q)$, kept.\label{fig:state1}}[0.47\linewidth]{
    \begin{tikzpicture}[baseline]
      \draw[qregion] (-1.85,-0.55) rectangle (1.85,1.05);
      \node[glabel, text=black!60] at (-1.45,0.75) {$Q$};
      \node[cvertex,label={[glabel]below:$a$}] (a) at (-0.9,0) {};
      \node[cvertex,label={[glabel]below:$b$}] (b) at (0.9,0) {};
      \node[tvertex,label={[glabel]above:$s$}] (s) at (0,0.55) {};
      \draw[gedge] (a) -- (s) -- (b);
      \draw[gedge] (a) -- ++(200:0.7); \draw[gedge] (b) -- ++(340:0.7);
      \node[glabel] at (0,-0.9) {$a\,s\,b$ path inside tree $Q$};
    \end{tikzpicture}}
  \end{subcaptionbox}
  \hfill
  \begin{subcaptionbox}{Case (2): $e$ on a residual cycle, kept.\label{fig:state2}}[0.47\linewidth]{
    \begin{tikzpicture}[baseline]
      \draw[qregion] (-2.0,-0.55) rectangle (-0.95,1.05);
      \node[glabel, text=black!60] at (-1.475,0.25) {$Q$};
      \node[cvertex,label={[glabel]below:$a$}] (a) at (-0.35,0) {};
      \node[cvertex,label={[glabel]below:$b$}] (b) at (1.3,0) {};
      \node[tvertex,label={[glabel]above:$s=x$}] (s) at (0.475,0.55) {};
      \draw[gedge] (a) -- (s) -- (b);
      \draw[gedge] (a) to[bend right=42] (b);
      \node[glabel] at (0,-0.9) {$x=s$ exposed, outside $Q$};
    \end{tikzpicture}}
  \end{subcaptionbox}

  \medskip

  \begin{subcaptionbox}{Case (3): $e\in M_J$, endpoint $a\in V(Q)$.\label{fig:state3}}[0.47\linewidth]{
    \begin{tikzpicture}[baseline]
      \draw[qregion] (-1.85,-0.55) rectangle (-0.35,1.05);
      \node[glabel, text=black!60] at (-1.45,0.75) {$Q$};
      \node[cvertex,label={[glabel]below:$a$}] (a) at (-0.9,0) {};
      \node[cvertex,label={[glabel]below:$b$}] (b) at (0.9,0) {};
      \node[tvertex,label={[glabel]above:$s$}] (s) at (0,0.55) {};
      \draw[gedge] (a) -- (s);
      \draw[medge] (s) -- (b);
      \draw[gedge] (a) -- ++(200:0.7);
      \node[glabel] at (0,-0.9) {match $sb$, attach $as$ to $Q$};
    \end{tikzpicture}}
  \end{subcaptionbox}
  \hfill
  \begin{subcaptionbox}{Case (4): $e\in M_J$, neither endpoint in $Q$.\label{fig:state4}}[0.47\linewidth]{
    \begin{tikzpicture}[baseline]
      \draw[qregion] (-2.0,-0.55) rectangle (-0.95,1.05);
      \node[glabel, text=black!60] at (-1.475,0.25) {$Q$};
      \node[cvertex,label={[glabel]below:$a$}] (a) at (-0.35,0) {};
      \node[cvertex,label={[glabel]below:$b=x$}] (b) at (1.3,0) {};
      \node[tvertex,label={[glabel]above:$s$}] (s) at (0.475,0.55) {};
      \draw[medge] (a) -- (s);
      \draw[gedge] (s) -- (b);
      \draw[gedge] (b) -- ++(340:0.65);
      \node[glabel] at (0,-0.9) {match $as$; $x=b$ exposed outside $Q$};
    \end{tikzpicture}}
  \end{subcaptionbox}
  \caption{The four cases when restoring the suppressed path $a\,{-}\,s\,{-}\,b$.
  Heavy edges are matching edges, and the dashed region indicates the tree
  component $Q$. In cases~(2) and~(4) a single exposed vertex $x$ lies outside
  $Q$; these cases are treated with \cref{lem:first-entry}.}
  \label{fig:restore-states}
\end{figure}

\section{Proof of the 3-Decomposition Conjecture}\label{sec:proof-3dc}

In this section we deduce \cref{thm:simple-fragile} from
\cref{thm:terminal-selection}, extend it to \cref{cor:multigraph-fragile}, and
derive 2DC and \cref{thm:3dc-resolved} from the corollary.

\begin{proof}[Proof of \cref{thm:simple-fragile}]
  Suppose \cref{thm:simple-fragile} is false, and choose a counterexample $G$
  with the fewest vertices. If $G$ is a tree, then $E(G)=E(G)\dcup\empt$ is
  already the required partition; thus $G$ contains a cycle.

  First, $G$ has no bridge. If $b$ is a bridge and $G_1,G_2$ are the components
  of $G-b$, then each $G_i$ is a smaller connected simple fragile subcubic
  graph: a cycle whose edge deletion left $G_i$ connected would, after
  restoring $b$ and the other component, leave $G$ connected. By minimality,
  decompose each $G_i$ into a spanning tree $T_i$ and a matching $M_i$. Then
  \[
    E(T_1)\dcup\{b\}\dcup E(T_2)
  \]
  is the edge set of a spanning tree of $G$, while $M_1\dcup M_2$ is a matching
  because $V(G_1)$ and $V(G_2)$ are disjoint. This is a contradiction.

  We may therefore assume that $G$ is bridgeless. Since $G$ is connected,
  bridgeless, and subcubic, and contains a cycle, every vertex of $G$ has
  degree two or three. Put $k=\abs{\Vtwo(G)}$.
  Before applying \cref{thm:terminal-selection}, we rule out $k\le 1$. In that
  case choose a vertex $v$, taking $v$ to be the terminal if $k=1$, and let $Q$
  be the one-vertex subgraph $\{v\}$. Every vertex outside $Q$ has degree
  three. Since $Q$ has no edges, every cycle of $G$ is edge-disjoint from $Q$,
  and $G$ contains a cycle. \Cref{thm:protected-cycle} therefore supplies a
  cycle $C$ for which $G-E(C)$ is connected, contrary to fragility.

  We therefore have $k\ge 2$. Apply \cref{thm:terminal-selection} to $G$. It
  yields a matching $M$ such that $G-M$ consists of a tree $Q$ containing all
  terminals and some residual cycles. If a residual cycle exists, it is
  edge-disjoint (indeed vertex-disjoint) from $Q$, and every vertex outside $Q$
  has degree three in $G$, because all degree-two vertices lie in $Q$.
  \Cref{thm:protected-cycle} then supplies a cycle $C$ edge-disjoint from $Q$
  for which $G-E(C)$ is connected, contrary to fragility. Therefore there is no
  residual cycle: the tree $Q$ is spanning, and $E(G)=E(Q)\dcup M$, contrary to
  the choice of $G$.
\end{proof}

We now extend \cref{thm:simple-fragile} to loopless multigraphs.

\begin{proof}[Proof of \cref{cor:multigraph-fragile}]
  We argue by induction on $\abs{E(G)}$. If $G$ has no parallel edges, then $G$
  is a simple graph and \cref{thm:simple-fragile} applies. Otherwise let $e$
  and $f$ be two parallel edges, both with ends $u$ and $v$, and put
  \[
    G'=G-f .
  \]

  The graph $G'$ is connected, because $e$ still joins $u$ and $v$, and
  subcubic. It is also fragile: a cycle $D$ of $G'$ is a cycle of $G$, so
  $G-E(D)$ is disconnected, and $G'-E(D)$ is a spanning subgraph of $G-E(D)$,
  hence disconnected as well.

  Moreover, every edge of $G'$ incident with $u$ or $v$ is a bridge of $G'$.
  Indeed, $\{e,f\}$ is a cycle of $G$, so by fragility
  \[
    G-\{e,f\}=G'-e
  \]
  is disconnected; thus $e$ is a bridge of $G'$. Next, $e$ and $f$ are both
  incident with $u$ and $d_G(u)\le 3$, so $d_{G'}(u)\le 2$. Hence, if $g\ne e$
  is an edge of $G'$ incident with $u$, then $e$ and $g$ are the only edges of
  $G'$ at $u$. A cycle of $G'$ through $g$ would use two edges at $u$, hence
  also $e$, contradicting that $e$ is a bridge of $G'$. So $g$ is a bridge of
  $G'$. The same argument applies at $v$.

  By the induction hypothesis, $E(G')=E(T)\dcup M$ for a spanning tree $T$ and
  a matching $M$ of $G'$. All bridges of $G'$ lie in $T$, so no edge of $M$ is
  incident with $u$ or $v$; that is, both ends of $f$ are exposed by $M$.
  Consequently $M\dcup\{f\}$ is a matching of $G$, $T$ is a spanning tree of
  $G$, and
  \[
    E(G)=E(T)\dcup\bigl(M\dcup\{f\}\bigr).
  \]
  This completes the induction.
\end{proof}

Next we derive 2DC in its standard loopless-multigraph form.

\begin{corollary}[2-Decomposition Conjecture]\label{cor:2dc-resolved}
  Every finite connected fragile loopless multigraph in which every vertex has
  degree two or three admits an edge partition into a spanning tree and a
  nonempty matching.
\end{corollary}

\begin{proof}
  Let $G$ be such a multigraph. By the handshaking identity,
  \[
    2\abs{E(G)}=2\abs{\Vtwo(G)}+3\abs{\Vthree(G)}=2\abs{V(G)}+\abs{\Vthree(G)},
  \]
  so every complementary spanning tree leaves
  \begin{equation}\label{eq:matching-nonempty}
    \abs{E(G)}-\bigl(\abs{V(G)}-1\bigr)=\frac{\abs{\Vthree(G)}}{2}+1\ge 1
  \end{equation}
  edges in the matching part. Applying \cref{cor:multigraph-fragile} gives
  $E(G)=E(T)\dcup M$ with $T$ a spanning tree and $M$ a matching, and
  $M\ne\empt$ by \cref{eq:matching-nonempty}.
\end{proof}

Finally, we prove \cref{thm:3dc-resolved}.

\begin{proof}[Proof of \cref{thm:3dc-resolved}]
  Let $G$ be a connected cubic loopless multigraph. Among all 2-regular
  subgraphs $C$ for which
  \[
    K:=G-E(C)
  \]
  is connected, choose one inclusion-maximal. The admissible family is
  nonempty because the empty subgraph is admissible.

  The graph $K$ is a finite connected subcubic loopless multigraph; we show
  that it is fragile. Every vertex of $C$ has degree one in $K$, because $G$ is
  cubic and $C$ uses exactly two of its incident edges, whereas every vertex of
  a cycle of $K$ has degree at least two in $K$. Hence
  \begin{equation}\label{eq:VC-disjoint}
    \text{every cycle of $K$ is vertex-disjoint from $C$.}
  \end{equation}
  Suppose that some cycle $D$ of $K$ is not separating, so that $K-E(D)$ is
  connected. By \cref{eq:VC-disjoint}, $C\cup D$ is a 2-regular subgraph of
  $G$. It strictly contains $C$, because $D$ is nonempty and edge-disjoint from
  $C$, and
  \[
    G-E(C\cup D)=K-E(D)
  \]
  is connected. This contradicts the maximal choice of $C$, so $K$ is fragile.

  By \cref{cor:multigraph-fragile}, $E(K)=E(T)\dcup M$ for a spanning tree $T$
  of $K$ and a matching $M$. Since $V(K)=V(G)$, the tree $T$ is a spanning tree
  of $G$, and
  \[
    E(G)=E(T)\dcup E(C)\dcup M
  \]
  is a $3$-decomposition of $G$.
\end{proof}

\section{Concluding remarks}\label{sec:scope}

We have shown that every connected cubic loopless multigraph has a
$3$-decomposition. We close with two open problems.

\begin{question}\label{q:finer}
  In a $3$-decomposition of a connected cubic graph on $n$ vertices, what is
  the extremal behaviour of natural parameters of the decomposition, such as
  the minimum number of cycle components in the $2$-regular part or the maximum
  number of distinct $3$-decompositions as a function of $n$?
\end{question}

\begin{question}\label{q:higher-degree}
  Lyngsie and Merker~\cite{LyngsieMerker2019} proved that every connected
  graph decomposes into a spanning tree, an even subgraph, and a star
  forest, and showed~\cite[Theorem~3]{LyngsieMerker2019} that for every
  $r\ge 4$ there exist connected $r$-regular simple graphs in which the star
  forest cannot be replaced by a matching. What connectivity or structural
  conditions on an $r$-regular graph ($r\ge 4$) guarantee a decomposition into
  a spanning tree, an even subgraph, and a matching?
\end{question}

\section*{Acknowledgements}

The author thanks Luyi Li and Erling Wei for helpful discussions
and suggestions. ChatGPT 5.6 was used during preliminary exploration
and language editing. The author developed the proofs and is
responsible for the final manuscript. A complete formalization
of the results of this paper in Lean~4 is available at
\url{https://github.com/mjc191812/The-3-Decomposition-Conjecture}.


\begin{thebibliography}{99}

\bibitem{AbdolhosseiniAkbariHashemiMoradian2016}
F.~Abdolhosseini, S.~Akbari, H.~Hashemi, and M.~S. Moradian,
\emph{Hoffmann-Ostenhof's conjecture for traceable cubic graphs},
preprint (2016), \texttt{arXiv:1607.04768}.

\bibitem{AboomahigirAhanjidehAkbari2018}
E.~Aboomahigir, M.~Ahanjideh, and S.~Akbari,
\emph{Decomposing claw-free subcubic graphs and $4$-chordal subcubic graphs},
preprint (2018), \texttt{arXiv:1806.11009}.

\bibitem{AkbariJensenSiggers2015}
S.~Akbari, T.~R. Jensen, and M.~Siggers,
\emph{Decompositions of graphs into trees, forests, and regular subgraphs},
Discrete Math. \textbf{338} (2015), no.~8, 1322--1327.

\bibitem{BachtlerHeinrich2023}
O.~Bachtler and I.~Heinrich,
\emph{Reductions for the 3-Decomposition Conjecture},
Procedia Comput. Sci. \textbf{223} (2023), 96--103 (full version at
\texttt{arXiv:2104.15113}).

\bibitem{BachtlerKrumke2022}
O.~Bachtler and S.~O. Krumke,
\emph{Towards obtaining a 3-decomposition from a perfect matching},
Electron. J. Combin. \textbf{29} (2022), no.~4, Paper 4.23.

\bibitem{BorseWaphare2009}
Y.~M. Borse and B.~N. Waphare,
\emph{On removable cycles in connected graphs},
J. Indian Math. Soc. (N.S.) \textbf{76} (2009), no.~1--4, 31--46.

\bibitem{BorseWaphare2013}
Y.~M. Borse and B.~N. Waphare,
\emph{Removable cycles avoiding two connected subgraphs},
ISRN Discrete Math. \textbf{2013} (2013), Article ID 164535.

\bibitem{BotlerJimenezSambinelliWakabayashi2024}
F.~Botler, A.~Jim\'enez, M.~Sambinelli, and Y.~Wakabayashi,
\emph{On the structure of a smallest counterexample and a new class verifying
the 2-Decomposition Conjecture},
Graphs Combin. \textbf{40} (2024), Paper 102.

\bibitem{BoydIwataTakazawa2013}
S.~Boyd, S.~Iwata, and K.~Takazawa,
\emph{Finding 2-factors closer to TSP tours in cubic graphs},
SIAM J. Discrete Math. \textbf{27} (2013), no.~2, 918--939.

\bibitem{FanGuoZhou2026}
G.~Fan, S.~Guo, and C.~Zhou,
\emph{The 3-Decomposition Conjecture of cubic graphs},
Sci. China Math. \textbf{69} (2026), no.~8, 2239--2248.

\bibitem{FanZhou2025}
G.~Fan and C.~Zhou,
\emph{Hoffmann-Ostenhof's 3-decomposition conjecture},
Discrete Math. \textbf{348} (2025), no.~7, Article 114454.

\bibitem{HoffmannOstenhof2011}
A.~Hoffmann-Ostenhof,
\emph{Nowhere-zero flows and structures in cubic graphs},
Ph.D. thesis, University of Vienna, 2011.

\bibitem{HoffmannOstenhofKaiserOzeki2018}
A.~Hoffmann-Ostenhof, T.~Kaiser, and K.~Ozeki,
\emph{Decomposing planar cubic graphs},
J. Graph Theory \textbf{88} (2018), no.~4, 631--640.

\bibitem{HongLiuYu2020}
Y.~Hong, Q.~Liu, and N.~Yu,
\emph{Edge decomposition of connected claw-free cubic graphs},
Discrete Appl. Math. \textbf{284} (2020), 246--250.

\bibitem{LiCui2014}
R.~Li and Q.~Cui,
\emph{Spanning trees in subcubic graphs},
Ars Combin. \textbf{117} (2014), 411--415.

\bibitem{LiuLi2020}
P.~Li and W.~Liu,
\emph{Decompositions of cubic traceable graphs},
Discuss. Math. Graph Theory \textbf{40} (2020), no.~1, 35--49.

\bibitem{LyngsieMerker2019}
K.~S. Lyngsie and M.~Merker,
\emph{Decomposing graphs into a spanning tree, an even graph, and a star
forest},
Electron. J. Combin. \textbf{26} (2019), no.~1, Paper 1.33.

\bibitem{Malkevitch1979}
J.~Malkevitch,
\emph{Spanning trees in polytopal graphs},
Ann. New York Acad. Sci. \textbf{319} (1979), 362--367.

\bibitem{OzekiYe2016}
K.~Ozeki and D.~Ye,
\emph{Decomposing plane cubic graphs},
European J. Combin. \textbf{52} (2016), 40--46.

\bibitem{Petersen1891}
J.~Petersen,
\emph{Die Theorie der regul\"aren Graphs},
Acta Math. \textbf{15} (1891), 193--220.

\bibitem{ProttiSouza2018}
F.~Protti and U.~S. Souza,
\emph{Decycling a graph by the removal of a matching: new algorithmic and
structural aspects in some classes of graphs},
Discrete Math. Theor. Comput. Sci. \textbf{20} (2018), no.~2, Paper 15.

\bibitem{Schonberger1935}
T.~Sch\"onberger,
\emph{Ein Beweis des Petersenschen Graphensatzes},
Acta Sci. Math. (Szeged) \textbf{7} (1935), 51--57.

\bibitem{ThomassenToft1981}
C.~Thomassen and B.~Toft,
\emph{Non-separating induced cycles in graphs},
J. Combin. Theory Ser. B \textbf{31} (1981), no.~2, 199--224.

\bibitem{Tutte1963}
W.~T. Tutte,
\emph{How to draw a graph},
Proc. London Math. Soc. (3) \textbf{13} (1963), 743--767.

\bibitem{XieZhouZhou2020}
M.~Xie, C.~Zhou, and S.~Zhou,
\emph{Decomposition of cubic graphs with a 2-factor consisting of three cycles},
Discrete Math. \textbf{343} (2020), 111839.

\bibitem{ZhangSzeider2025}
T.~Zhang and S.~Szeider,
\emph{The 3-Decomposition Conjecture: a SAT-based approach with specialized
propagators},
in \emph{31st International Conference on Principles and Practice of Constraint
Programming (CP 2025)}, LIPIcs \textbf{340}, 2025, pp. 39:1--39:19.

\end{thebibliography}
\end{document}